\documentclass[a4paper,11pt]{amsproc}
\usepackage{amsmath}
\usepackage{amssymb}
\usepackage{amsthm}
\usepackage{mathrsfs}
\usepackage{soul}
\theoremstyle{plain}
\usepackage{amsmath,amsthm,amscd,amsfonts,amssymb}
\usepackage{graphicx}
\usepackage{color}
\usepackage[colorlinks]{hyperref}

\newtheorem{theorem}{Theorem}[section]

\newtheorem{lemma}[theorem]{Lemma}

\theoremstyle{definition}

\newtheorem{question}[theorem]{Question}
\newtheorem{example}[theorem]{Example}
\theoremstyle{remark}

\numberwithin{equation}{section}
\theoremstyle{definition}
\theoremstyle{definition}
\theoremstyle{definition}
\theoremstyle{definition}
\theoremstyle{definition}
\theoremstyle{definition}
\theoremstyle{remark}
\theoremstyle{definition}
\theoremstyle{definition}

\newcommand{\cal}{\mathcal}

\newcommand{\R}{\mathbb{R}}

\newcommand{\mxa}{\mathcal{M}(X,\mathcal{A})}

\newcommand{\liinf}{L_I^\infty(\mu)}

\newcommand{\umuip}{U_{\mu_I}^+}
\newcommand{\umuipp}{U_{\mu_I}^{++}}
\newcommand{\zi}{  Z[I] }

\newcommand{\mmui}{m_{\mu_I} }
\newcommand{\umui}{U_{\mu_I} }
\newcommand{\imuxa}{I\mu(X,\cal{A})}
\begin{document}	
	%% The title of the paper goes here.  Edit your title.
	\title[A Generalization of $ m $-topology and $ U $-topology on $ \mxa $]{A Generalization of $ m $-topology and $ U $-topology on rings of measurable functions}
			\author{Soumyadip Acharyya}
	\address[Soumyadip Acharyya]{Division of Science, Mathematics, and Engineering, University of South Carolina Sumter, 200 Miller Road,
		Sumter, SC 29150, USA }
	\email {acharyya@mailbox.sc.edu}

	\author{Rakesh Bharati}
	\address[Rakesh Bharati]{Department of Pure Mathematics, University of Calcutta, 35, Ballygunge Circular Road, Kolkata - 700019, INDIA} 
	\email{rbharati.rs@gmail.com}

	\author{Atasi Deb Ray }
	\address[A. Deb Ray]{Department of Pure Mathematics, University of Calcutta, 35, Ballygunge Circular Road, Kolkata - 700019, INDIA} 
	\email{debrayatasi@gmail.com}

	\author{ Sudip Kumar Acharyya}
	\address[Sudip Kumar Acharyya]{Department of Pure Mathematics, University of Calcutta, 35, Ballygunge Circular Road, Kolkata - 700019, INDIA}
	\email{sdpacharyya@gmail.com}

	\thanks{The first author acknowledges that this research was partially supported by the University of South Carolina Sumter 2021 Summer Research Stipend Program. }

	\thanks{The second author acknowledges financial support from University Grants Commission, New Delhi, for the award of research fellowship (F.No. 16-9(June 2018)/2019 (NET/CSIR))}

	\begin{abstract}
		For a measurable space ($X,\mathcal{A}$), let $\mathcal{M}(X,\mathcal{A})$ be the corresponding ring of all real valued measurable functions and let $\mu$ be a measure on ($X,\mathcal{A}$). In this paper, we generalize the so-called $m_{\mu}$ and $U_{\mu}$ topologies on $\mathcal{M}(X,\mathcal{A})$ via an ideal $I$ in the ring $\mathcal{M}(X,\mathcal{A})$. The generalized versions will be referred to as the $m_{\mu_{I}}$ and $U_{\mu_{I}}$ topology, respectively, throughout the paper. $L_{I}^{\infty} \left(\mu\right)$ stands for the subring of $\mathcal{M}(X,\mathcal{A})$ consisting of all functions that are essentially $I$-bounded (over the measure space ($X,\mathcal{A}, \mu$)). Also let $I_{\mu} (X,\mathcal{A}) = \big \{ f \in \mathcal{M}(X,\mathcal{A}) : \, \text{for every} \, g \in \mathcal{M}(X,\mathcal{A}), fg \, \, \text{is essentially} \, I$-$\text{bounded} \big \}$. Then $I_{\mu} (X,\mathcal{A})$ is an ideal in $\mathcal{M}(X,\mathcal{A})$ containing $I$ and contained in $L_{I}^{\infty} \left(\mu\right)$. It is also shown that $I_{\mu} (X,\mathcal{A})$ and $L_{I}^{\infty} \left(\mu\right)$ are the components of $0$ in the spaces $m_{\mu_{I}}$ and $U_{\mu_{I}}$, respectively. Additionally, we obtain a chain of necessary and sufficient conditions as to when these two topologies coincide. In particular, it is proved that they coincide if and only if the three sets $\mathcal{M}(X,\mathcal{A})$, $L_{I}^{\infty} \left(\mu\right)$, and $I_{\mu} (X,\mathcal{A})$ coincide and this holds when and only when $\mathcal{M}(X,\mathcal{A})$, with either topology, is connected, which is further equivalent to each of these two topological spaces being a topological ring as well as a topological vector space. Furthermore, it is easy to verify that $L_{I}^{\infty} \left(\mu\right)$ is a pseudonormed linear space regardless of the choice of $I$. Assuming $I$ is a real maximal ideal in $\mathcal{M}(X,\mathcal{A})$, we show that $L_{I}^{\infty} \left(\mu\right)$ is complete. We construct two examples of spaces of the form $L_{I}^{\infty} \left(\mu\right)$, each of which is not only non-separable but is also based on a non-trivial ideal $I$. 
	\end{abstract}
	\maketitle
	
	\section{Introduction}
	We start with a measure space $ (X,\cal{A},\mu) $, where $ X $ is a non-empty set, $ \cal{A} $ is a $ \sigma $-algebra of subsets of $ X $ and $ \mu :\cal{A}\rightarrow [0,\infty]$ is a  measure. A function $ f:X\rightarrow \R $ is called $ \cal{A} $-measurable or simply measurable if for any $ \alpha\in \R $, $ f^{-1}(-\infty,\alpha) $ is a member of $ \cal{A} $. Members of $ \cal{A} $ are often called $ \cal{A} $-measurable sets or just measurable sets if there is no chance of any confusion. The collection $ \mxa $ of all measurable functions constitutes a commutative lattice ordered algebra with unity if the operations are defined pointwise on $ X $. It may be mentioned that a number of properties of $ \mxa $ and a few of its chosen subalgebras have already been investigated in \cite{ref2},\cite{ref4},\cite{ref5},\cite{adit}.
	
	Suppose for $ f\in\mxa $ and a positive unit $ u $ of this ring and $\epsilon>0$ in $ \R $, $ m_\mu(f,u)=\{g\in\mxa:|f(x)-g(x)|<u(x)  $ for all $ x $, almost every where on $ X $ with respect to the measure $ \mu $\} and $ U_\mu(f,\epsilon)=\{g\in\mxa: \sup_{x\in X\setminus A_g}|f(x)-g(x)|<\epsilon $ for some $ \cal{A} $-measurable set $ A_g $ with $ \mu(A_g)=0 \}$. Then \{$ m_\mu(f,u):f\in \mxa, u,$ a positive unit in $ \mxa $\} is an open base for a uniquely determined topology, which is called $ m_\mu $-topology on $ \mxa $ in \cite{ref3}. $ \mxa $ with the $ m_\mu $-topology becomes a topological ring. Incidentally the family $ \{U_\mu(f,\epsilon): f\in \mxa,\epsilon>0\} $ turns out to be an open base for a topology called $ U_\mu $-topology on $ \mxa $ in \cite{ref2}. The $ U_\mu $-topology is coarser than the $ m_\mu $-topology on $ \mxa $ and $ \mxa $ equipped with the $ U_\mu $-topology is an additive topological group, which is in general neither a topological ring nor a topological vector space. However it is established in \cite{ref2}, Theorem 3.7 that $ \mxa $ is a topological ring in the $ U_\mu $-topology if and only if it is a topological vector space with respect to the same topology and this precisely happens when and only when each function in $ \mxa $ is essentially bounded on $ X $ with respect to the measure $ \mu $. In the present article our intention is to initiate a further generalization of the $ m_\mu $-topology and the $ U_\mu $-topology both via an ideal  $ I $ in the ring $ \mxa $ and to place the last mentioned facts on a wider setting. To be specific with respect to any such ideal $ I $ in $ \mxa $, we call a function $ f \in \mxa $ a $ \mu_I $-unit if there exists $ Z\in \zi=\{Z(g):g\in I\} $ such that $ \mu\{x\in Z:f(x)=0\}=0 $. Let $ \umui $ be the family of all $ \mu_I $-units in $ \mxa $ and $ \umuip=\{u\in \umui: $ there exists $ Z\in \zi $ and $ E\subseteq Z $ with $ \mu(E)=0 $ such that $ u(x)>0$\}. Let for $ f\in\mxa,\ u\in\umuip $ and $ \epsilon>0 $, $\umui(f,\epsilon)=\{g\in\mxa:\sup_{x\in Z\setminus E}|f(x)-g(x)|<\epsilon $ for some $ Z\in \zi $ and $ E\subseteq Z $ with $ \mu(E)=0 $\} and  $\mmui(f,u)= \{ g\in\mxa:$ there exist $Z\in\zi$ such that $|g(x)-f(x)|<u(x)$ for all $ x\in Z $ almost everywhere\}. It needs some routine computations to prove that the family \{$ \mmui(f,u): f\in\mxa, u\in\umuip $\} is an open base for some topology which we wish to designate as $ \mmui $-topology on $ \mxa $. Furthermore, it is not at all hard to establish that $ \mxa $ becomes a topological ring in the $ \mmui $-topology. Incidentally the family $ \{\umui(f,\epsilon):f\in\mxa,\epsilon>0\} $ also turns out to be an open base for a uniquely determined topology, which we call as $ \umui $-topology on $ \mxa $. $ \mxa $ equipped with the $ \umui $-topology is an additive topological group but it may not be a topological ring. With the special choice $ I=\{0\} $, the $ \mmui $-topology and the $ \umui $-topology reduce to $ m_\mu $-topology and $ U_\mu $-topology respectively described earlier.
	
	We let $ \imuxa =\{f\in\mxa:$ given $ g\in\mxa $ there exists $ Z\in\zi $ and a measurable subset $ E $ of $ Z $ with $ \mu(E) $=0 such that $ f.g $ is bounded on $ Z\setminus E \}$ and $L_I^{\infty}(X,\cal{A},\mu)\equiv \liinf=\{f\in\mxa:$ there exists $ Z\in\zi $ and $ E\subseteq Z $ with $ \mu(E)=0 $ such that $ f $ is bounded on $ Z\setminus E \}$. These two sets play a key role in enabling us to establish a number of main results in the present paper. It turns out that $ \imuxa $ is an ideal in $ \mxa $ and $ I\subseteq \imuxa \subseteq \liinf$. We realize that $ \imuxa $ and $ \liinf $ are the components of $ 0 $ in $ \mxa$ equipped with the $\mmui $-topology and the $\umui$-topology respectively. With the special choice $ I=\{0\} $, the second fact in the last assertion reads: the space $ L^{\infty}(\mu) $ of all essentially bounded measurable functions over the measure space $ (X,\cal{A},\mu) $ is just the component of the function $ 0 $ in the topoloical group $ \mxa $ equipped with the $ U_\mu $-topology. The set $ \imuxa $, though not in general the component of $ 0 $ in $ \mxa $ in the $ \umui $-topology, happens to be the maximal connected ideal in this ring with the same topology. Before narrating the organization of the technical results in this article, we like to mention that similar kind of topologies viz. $ \mmui $-topology and $ \umui $-topology, as generalizations of the well-known $ m $-topology and the topology of uniform convergence also called the $ U $-topology on the ring $ C(X) $ of all real-valued continuous functions on a completely regular Hausdorff topological space $ X $, via an ideal $ I $ in $ C(X) $ are initiated in \cite{ref1} and \cite{ref6} respectively.
	
	In the technical section 2 of this article we establish several facts related to the $ \mmui $-topology on $ \mxa $. As in \cite{ref1}, for any $ f\in \mxa $, the map \begin{alignat*}{2}
				\Phi_f:\R &\rightarrow \mxa\\
	  r&\mapsto r.f
	\end{alignat*}
	is introduced. It is proved that for any $ f\in \imuxa$, $ \Phi_f $ is a continuous map if the $ \mmui $-topology is imposed on $ \mxa $ (Theorem \ref{t-2.4}). This fact together with the result that $ \liinf $ is a clopen subset of $ \mxa $ in the $ \mmui $-topology (Theorem \ref{t-2.2}) yields that $ \imuxa $ is the component of $ 0 $ in the $ \mmui $-topology (Theorem \ref{t-2.5}). We realize that $ \liinf $ becomes a pseudonormed linear space with respect to a natural definition of generalized essential supremum $ ||f||_I^{\infty} $ for functions $ f\in \liinf $. It is easy to verify that the pseudonorm topology on $ \liinf $ is weaker than the subspace topology on this set relative to the $ \mmui $-topology on $ \mxa $. These two topologies coincide if and only if the $ \mmui $-topology becomes connected/ locally connected. Indeed several other necessary and sufficient conditions (each equivalent to the coincidence of these two topologies) involving the sets $ \imuxa $ and $ \liinf $ are also found out (Theorem \ref{t-2.7}). This is followed by a complete description of the connected ideals in $ \mxa $ in the $ \mmui $-topology (Theorem \ref{t-2.8}).
	
	In the technical section 3 of this paper, our main focus is on proving a number of results concerning the $ \umui $-topology on $ \mxa $. Since the map $ \Phi_f $ is proved to be continuous for each $ f\in\imuxa $ in the $ \mmui $-topology, it follows that the same map is also continuous for each $ f\in \imuxa $ if $ \mxa $ is given the $ \umui $-topology as this topology is weaker than the $ \mmui $-topology. However we make a minor improvement of the last assertion by proving that the map $ \Phi_f $ is continuous for each $ \liinf $ in the $ \umui $-topology (Theorem \ref{t-3.1}). This has ultimately led to the fact that $ \liinf $ is the component of $ 0 $ in the space $ \mxa $ with the $ \umui $-topology (Theorem \ref{t-3.3}). It is well-known that the component of $ 0 $ in a topological ring is necessarily an ideal in this ring (Theorem 2.6, \cite{ref6}). However if $ R $ is a ring with a topology $ \tau $ for which $ (R,\tau) $ is just an additive topological group without being a topological ring, then the component of $ 0 $ in $ R $ may not be an ideal in $ R $. A counterexample for this purpose in rings of continuous functions $ C(X) $ is already provided in \cite{ref6}. Here in this article we provide yet another counterexample in rings of measurable functions by choosing  appropriately $ X, \cal{A},\mu $ and the ideal $ I $ in $ \mxa $ (Example \ref{e-3.4}). It is already mentioned in this introductory section that $ \mxa $ equipped with the $ \umui $-topology is in general neither a topological ring nor a topological vector space. It is realized however that the $ \umui $-topology makes $ \mxa $ a topological ring if and only if it renders $ \mxa $ a topological vector space. Such things happen precisely when (amongst other necessary and sufficient conditions) the $ \umui $-topology coincides with the $ \mmui $-topology (Theorem \ref{t-3.6}).

In the technical section 4 of the paper we investigate mainly a few problems concerning the pseudonorm topology of the space $ \liinf $. It is realized that $ \liinf $ is a complete pseudonormed linear space if the filter $ \zi $ of $ \cal{A} $-measurable sets is closed under countable intersection (Theorem \ref{t-4.2}), in particular if $ I $ is a real maximal ideal in $ \mxa $. It is folklore that the space $ L^{\infty}(\mu) $ of all essentially bounded measurable functions over a measure space $ (X,\cal{A},\mu) $ are mostly non-separable (see the comments in page 580 in \cite{ref7}). In the fitness of things, by making suitable special choices of $ X,\cal{A}, \mu $ and non-zero ideals $ I $ in the rings $ \mxa $, we construct two different complete pseudonormed linear spaces $ \liinf $ none of which is separable (Example \ref{e-4.3}, Example \ref{e-4.4}). We conclude this article after showing that with the special choice $ \mu= $ counting measure on the measurable space $ (X,\cal{A}) $ that for two ideals $ I $ and $ J $ in $ \mxa $, $ \mmui $-topology = $ m_{\mu_J} $ topology if and only if $ \umui $-topology = $ U_{\mu_J} $-topology if and only if $ I=J $. We then check that on choosing $ \mu= $ an appropriately defined Dirac measure on $ X $, for any two ideals $ I, J $ in $ \mxa $, $ \umui $-topology = $ U_{\mu_J} $-topology and $ \mmui $-topology = $ m_{\mu_J} $-topology. We like to mention that various other properties enjoyed by the $ m $-topology on the ring $ \mxa $ has been established recently in \cite{adit}.
	
	\section{$ \mmui $-topology on $ \mxa $}
\begin{theorem}\label{t-2.1}
	The set $ \umui $ is open in $ \mxa$ in the $ \mmui $-topology.
\end{theorem}	
	\begin{proof}
		Since $ \umui =\{u\in\mxa:$ there exists $ Z\in \zi $ and $ E\subseteq Z $ with $ \mu(E)=0$, such that $ |u(x)|>0 $ for each $ x\in Z\setminus E \}$, it suffices to show that $ \umuip $ is open in the $ \mmui $-topology. This is indeed the case because for $ u\in\umuip $, it is not hard to check that $\mmui(u,\frac{u}{2})\subseteq\umuip $, from which it follows that $ u $ is an interior point of $ \umuip $ in the $ \mmui $-topology.
	\end{proof}
	\begin{theorem}\label{t-2.2}
		$ \liinf $ is a clopen subset of $ \mxa $ in the $ \mmui $-topology.
	\end{theorem}
\begin{proof}
	That $ \liinf $ is open in $ \mxa $ in the $ \mmui $-topology follows from the easily verifiable result that for $ f\in\liinf $, $ \mmui(f,1)\subseteq \liinf $. On the other hand if $ g\in\mxa\setminus \liinf$, then it is easy to check that $ \mmui(g,1)\cap \liinf=\emptyset $. It follows that $ \liinf $ is a closed set in $ \mxa $ in the $ \mmui $-topology.
\end{proof}
The following result needs only some routine  computations for its proof.
\begin{theorem}\label{t-2.3}
	$ \imuxa $ is an ideal in the ring $ \mxa $ and $ I\subseteq \imuxa\subseteq\liinf $.
\end{theorem}
In the context of rings of continuous functions $ C(X) $, where $ X $ is a completely regular Hausdorff topological space, for $ f\in C(X) $ the map \begin{alignat*}{2}
	\Phi_f:\R &\rightarrow C(X)\\
	r&\mapsto r.f
\end{alignat*}
is introduced. In the present paper we exploit a few of the properties of an analogous map defined by the same notation $ \Phi_f:\R\rightarrow \mxa $ defined as $ r\mapsto r.f $ to be useful for our later results.
\begin{theorem}\label{t-2.4}
	For each function $ f\in\imuxa $, the map $ \Phi_f $ is continuous if $ \mxa $ is given the $ \mmui $-topology.  $ ( $A special case of this result with $ I=\{0\} $ and $ \mu =$ counting measure is already proved in \cite{adit}$ ) $. 
\end{theorem}
\begin{proof}
	Let $ f\in \imuxa $ and $ r\in \R $ be arbitrary. Then a typical basic open neighbourhood of $ r.f $ in $ \mxa $ in the $ \mmui $-topology is of the form $ \mmui(rf,u) $, where $ u\in \umuip $. Accordingly there exists $ Z\in \zi $ and $ A\in \cal{A} $ with $ A\subseteq Z $ and $ \mu(A)=0 $ for which $ u(x)>0 $ for each $ x $ in $ Z\setminus A $. Set $ g=|u|+\chi_{(X\setminus Z)\cup A} $, here $ \chi_E $ stands for the characteristic function of a measurable subset $ E $ of $ X $. Then $ g $ is a positive multiplicative unit in the ring $ \mxa $. Therefore $ \frac{1}{g}\in\mxa $. The hypothesis $ f\in \imuxa $ implies that $ \frac{f}{g} $ is bounded almost everywhere (with respect to the measure $\mu$) on some zero-set $ Z^{'} \in\zi$. This means that there exists $ B\in\cal{A} $ with $ \mu(B)=0 $ such that $ B\subseteq Z^{'} $ for which $ |\frac{f}{g}|\leq\lambda $ on $ Z^{'}\setminus B $ for some $\lambda>0$ in $ \R $. Let $ s\in (r-\frac{1}{\lambda}, r+\frac{1}{\lambda}) $. Then on the set $ Z\cap Z^{'} \setminus (A\cup B) $, $ |\phi_f(r)-\phi_f(s)|=|s-r||f|\leq\frac{1}{\lambda} .|\frac{f}{g}|.|g|$$\leq\frac{1}{\lambda}.\lambda.|u|=|u|$. Since $ Z\cap Z^{'}\in \zi $ and $ \mu(A\cup B)=0 $. It follows that $ \phi_f(s)\in \mmui(\phi_f(r),u) $. This ensures the continuity of the map $\phi_f$ at the point $ r\in \R $, chosen arbitrarily. We are now ready to prove the first main technical result in this section.
\end{proof}
\begin{theorem}\label{t-2.5}
	$ \imuxa $ is the component of $ 0 $ of the ring $ \mxa $ in the $ \mmui $-topology.
\end{theorem}
\begin{proof}
	It is trivial that $ \imuxa=\bigcup\limits_{f\in\imuxa} \phi_f(\R)=$ the union of a family of connected subsets of $ \mxa $ in the $ \mmui $-topology each of which contains the function $ 0 $. Hence $ \imuxa $ is a connected ideal in $ \mxa $. To show that $ \imuxa $ is the component of $ 0 $ in  $ \mxa $ (equivalently the maximal connected ideal in $ \mxa $), if possible let $ J $ be a connected ideal in $ \mxa $ with $ \imuxa\subsetneq J $. Choose $ f\in J\setminus \imuxa $. Then there exists $ g\in \mxa $ such that for each $ Z\in \zi $ and for each $ M>0 $ in $ \R $, $ \mu\{t\in Z: |f(t)g(t)|>M\}>0 $. Consequently $ f.g\notin \liinf $. But we note that $ f.g\in J $, since $ J $ is an ideal in $ \mxa $. So we can write $ J=(J\cap\liinf)\cup (J\setminus\liinf) =$ the union of two non-empty disjoint open sets in the space $ J $. We exploit the clopenness of the set $ \liinf $ in $ \mxa $ in the $ \mmui $-topology to prove the last assertion. This contradicts the connectedness of $ J $ in $ \mxa $.
\end{proof}
An $ f\in \mxa $ is called essentially $ I$-bounded on $ X $ if there exists $ Z\in\zi $, a measurale subset $ E $ of $ Z $ with $ \mu(E)=0 $ and a real number $ \lambda >0$ such that $ |f(x)|\leq\lambda $ holds for all $ x\in Z\setminus E $. Any such $ \lambda $ is called an essential $ I $-upper bound of $ f $ over $ X $. With $ I=\{0\} $, essential $ I $-upper bound coincides with essential upper bound. Thus $ \liinf $ is precisely the family of all essentially $ I $-bounded functions (over $ X $) in the ring $ \mxa $. For each $ f\in \liinf $, let $A_f $ be the aggregate of all essential $ I $-upper bounds of $ f $ over $ X $. Suppose $ ||f||_I^{\infty} =$ infimum of the set $ A_f $. It is clear that $ ||f||_I^{\infty}\geq0 $, and if $ f=0 $ then $ ||f||_I^{\infty} =0$. Also if $ f,g\in \liinf $ and $ l\in A_f $ and $ m\in A_g $, then there exist $ Z_1 , Z_2\in\zi$ and measurable sets $ E_2,E_2 $ with $ E_1\subseteq Z_1,\ E_2\subseteq Z_2 $ such that $ \mu(E_1)=0=\mu(E_2) $ and $ |f(x)|\leq l $ for all $ x\in Z_1\setminus E_1 $ and $ |g(x)|\leq m $ for all $ x\in Z_2\setminus E_2 $. It follows that for all $ x\in (Z_1\cap Z_2)\setminus (E_1\cup E_2) $, $ |(f+g)(x)|\leq l+m $ with $ Z_1\cap Z_2\in Z[J] $ and $ \mu(E_1\cup E_2)=0 $. Thus $ l+m $ is an essential $ I $-upper bound of $ f+g $ over $ X $ and therefore $ ||f+g||_I^{\infty}\leq l+m $. Since $ l\in A_f $ and $ m\in A_g $ are chosen arbitrarily, this implies that $ ||f+g||_I^{\infty}\leq$ Inf$A_f +$ Inf$A_g=||f||_I^{\infty}+||g||_I^{\infty} $. Finally it is routine to check for any $ c\in\R $ and $ f\in\liinf $ that $ ||cf||_I^{\infty}=|c|||f||_I^{\infty} $. Thus we have already proved the following fact.
\begin{theorem}\label{t-2.6}
	$ (\liinf,||.||_I^{\infty}) $ is a pseudonormed linear space over $ \R $.
\end{theorem}
Since for $ f\in \liinf $ and $ \delta>0 $ in $ \R $, $ \mmui(f,\delta) \subseteq \{g\in \liinf: ||g-f||_I^{\infty}<\delta\}$, it follows that the pseudonorm topology on $ \liinf $ is weaker than the relative topology on $ \liinf $ induced by the $ \mmui $-topology on $ \mxa $.

The following theorem elaborates in precise terms, when these two topologies on $ \liinf $ coincide.
\begin{theorem}\label{t-2.7}
	The following statements are equivalent.\begin{enumerate}
		\item $ \liinf=\mxa $.
		\item$ \imuxa=\mxa $.
		\item The pseudonorm topology on $ \liinf $ is identical to the relative $ \mmui $-topology on it.
		\item The $ \mmui $-topology on $ \mxa $ is connected.
		\item The $ \mmui $-topology on $ \mxa$ is locally connected.
		\item $ \zi $ is a $ \mu $-stable set in the sense that given $ f\in\mxa $, there exists $ Z\in\zi $ and $ E\subseteq Z $ with $ \mu(E)=0 $ such that $ f $ is bounded on $ Z\setminus E $.
				\item $ \mxa $ with $ \mmui $-topology is a topological vector space.
	\end{enumerate}
\end{theorem}
\begin{proof}
	$ (1)\Rightarrow (3):$ Let (1) be true. In view of the last observation preceeding Theorem \ref{t-2.7}, it suffices to show that the relative $ \mmui $-topology on $ \liinf $ is weaker than the pseudonorm topology on this set. Choose $ f\in \liinf $ and $ u\in U_I^{+} $. Then there exists $ Z\in\zi $ and $ A\in \cal{A} $, $ A\subseteq Z $ with $ \mu(A)=0 $ such that $ u(x)>0 $ for each $ x\in Z\setminus A $.
	Set $ v(x)=u(x) $ if $ x\in Z\setminus A $ and $ v(x)=1 $, otherwise, $x\in X$. Then $ v\in \mxa $ and is a positive unit in the ring. So $ \frac{1}{v}\in \mxa=\liinf $. There exists therefore $ M>0 $, $ Z^{'}\in\zi $ and $ A^*\subseteq Z^{'} $ with $ \mu(A^*)=0 $ for which for each $ x\in Z^{'}\setminus A^*, |\frac{1}{v(x)}|\leq M $. This implies that for each $ x\in (Z\cap Z^{'})\setminus(A\cup A^*) $, $ u(x)=v(x)\geq \frac{1}{M}=\lambda $, say. Now choose $ g\in\liinf $ with $ ||g-f||_I^{\infty}<\lambda $. Then there exists $ Z_1\in\zi $ and $ A_1\subseteq Z_1 $ with $ \mu(A_1)=0 $ for which we can write $ |g(x)-f(x)|<\lambda $ for each $ x\in Z_1\setminus A_1 $. This implies that for each $ x\in (Z_1\cap Z\cap Z^{'})\setminus (A\cup A_1\cup A^*) $, $ |g(x)-f(x)|<u(x) $ and $ Z_1\cap Z\cap Z^{'}\in\zi $ and also $ \mu(A\cup A_1\cup A^*)=0 $. Thus it follows that $ g\in \mmui(f,u) $, consequently $ \{g\in \liinf:||g-f||_I^{\infty}<\lambda\}\subseteq \mmui(f,u) $.\\
	$ (3)\Rightarrow (1):$ Suppose $ (1) $ is false. Choose $ f\in\mxa $ such that $ f\notin \liinf $ and $ f\geq 1 $. Then for each $ Z\in\zi $ and $ n\in\mathbb{N} $, $ \mu\{t\in Z: f(t)>n\}>0 $ and hence $ \mu\{t\in Z:g(t)<\frac{1}{n}\}>0 $\ ... $ (*) $, on writing $ g(t)=\frac{1}{f(t)},\ t\in X $. Since $ g\in\liinf $, we have then $ \mmui(0,g)\subseteq \liinf $. Thus $ \mmui(0,g) $ is a neighbourhood of $ 0 $ in the relative $ \mmui $-topology on $ \liinf $. To show that $ (3) $ is false, we shall show that $ \mmui(0,g) $ is not a neighbourhood of $ 0 $ in the pseudonorm topology on $ \liinf $. Indeed for any $ \epsilon>0 $ and $ Z\in\zi $, it follows from the relation $ (*) $ above that $ \frac{\epsilon}{2}>g $ on each set $ Z\in\zi $ with positive $ \mu $-measure. Consequently $ \frac{\epsilon}{2}\notin\mmui(0,g) $. Thus $ \frac{\epsilon}{2}\in \{h\in\liinf:||h||_I^{\infty}<\epsilon\}\setminus \mmui(0,g)$ and we are through. Thus the statements $ (1) $ and $ (3) $ are equivalent.\\
	$ (1)\Rightarrow (4)$ is immediate because a pseudonorm topology is path connected and therefore connected.\\
 $ (4)\Rightarrow (1) $ follows from Theorem \ref{t-2.2}.\\
	$ (2)\Rightarrow (5) :$ Let $ (2) $ be true and $\epsilon>0$ be arbitrary. It suffices to show that $ \mmui(0,\epsilon) $ is connected in the $ \mmui $-topology. Indeed for each $ f\in\mxa $, by $ (2) $, $ f\in \imuxa $. Consequently by Theorem \ref{t-2.4}, $ \phi_f $ is a continuous map. Hence $ \mmui(0,\epsilon) =\bigcup\limits_{f\in\imuxa}\phi_f[0,1]=$ the union of a family of connected subsets of $ \mxa $ (in the $ \mmui $-topology), each of which contains the function $ 0 $. Thus $ \mmui(0,\epsilon) $ is connected in the $ \mmui $-topology.\\
	$ (5)\Rightarrow(2) $. Let $ (2) $ be false. Choose $ f\in \mxa\setminus\imuxa $. To show that $ (5) $ is false, it is sufficient to show that there does not exist any open connected subset of $ \mxa $ containing $ 0 $ (in the $ \mmui $-topology). If possible let $ G $ be an open connected subset of $ \mxa $ in the $ \mmui $-topology containing $ 0 $. Since from Theorem \ref{t-2.5}, $ \imuxa $ is the maximal connected subset of $ \mxa $ containing $ 0 $ in this topology, we can write $ 0\in \mmui(0,u)\subseteq G\subseteq \imuxa $ for some $ u\in\umuip $. Therefore $ u(x)>0 $ for each $ x\in Z\setminus A $ for some $ Z\in \zi $ and $ A\subseteq Z $ with $ \mu(A)=0 $.\\
	Define $ v:X\rightarrow \R $ as follows $ v(x)=u(x) $ if $ x\in Z\setminus A $ and $ v(x)=1 $ otherwise. Then $ v\in \mxa $ and is a positive unit in $ \mxa $. we observe that $ \frac{f}{1+|f|}.v\in\mmui(0,v)\subseteq \mmui(0,u)\subseteq G\subseteq\imuxa $. Since $ \imuxa $ is an ideal in $ \mxa $, it follows that $ f\in \imuxa $, a contradiction.\\
	$ (2)\Leftrightarrow (6) $ follows by employing routine arguments.\\
	$ (4)\Rightarrow (7) $: This follows from the equivalence $ (1)\Leftrightarrow (3) $ and $ (1)\Leftrightarrow (4) $ and the fact that any pseudonormed linear space is a topological vector space.\\
	$ (7)\Rightarrow (4) $:  Let $ \liinf\subsetneq\mxa $. Choose $ f\in\mxa $ such that $ f\notin \liinf $. Without loss of generality we can assume that $ f\geq 1 $ on $ X $. Now for all $z\in \zi $ and $n\in \mathbb{N} $, we should have $ \mu\{x\in Z: f(x)>n\}>0\Rightarrow \mu\{x\in Z:g(x)<\frac{1}{n}\}>0 $ where $ g=\frac{1}{f}\ldots (*)$. We claim that for any two $ r,s\in \R,\ r\neq0,\ s\neq0 $, the scalar multiplication function  \begin{alignat*}{2}
		\psi:\R\times \mxa&\rightarrow \mxa\\
		(\alpha,f)&\mapsto \alpha.f
	\end{alignat*}
	is not continuous at the point $ (r, \underline{s}) $, here $ \underline{s} $ is the function on $ X $ whose value is constantly equal to $ s $. If possible let $ \psi $ be continuous at $ (r,\underline{s}) $. Then $ \mmui(\underline{rs}, g) $ is an open neighbourhood of $ \underline{rs} $ in the $ \mmui $-topology in $ \mxa $. So $\exists$ $\lambda>0$ and a $ v\in\umuip $ such that $ (r-\lambda, r+\lambda).\mmui(\underline{s},v)\subseteq\mmui(\bar{rs},g) $. In particular for any $ t: r-\lambda<t<r+\lambda $, $ \underline{ts}\in \mmui(\underline{rs},g) $. Thus $ |\underline{ts-rs}|(x)<g(x)$ for all $x $ almost everywhere with respect to $\mu$ on some $ Z\in\zi $, i.e., $g(x)>|(t-r)s|$ for all $x\in Z$ almost everywhere with respect to $\mu$. Choose $ t\neq r $ anywhere in $ (r-\lambda,r+\lambda) $ then $ g(x)>(t-r)s$ for all $x\in Z$ almost everywhere with respect to $\mu$. This implies $ g(x)>\frac{1}{n}$ for all $x\in Z $ for some $ n\in \mathbb{N} $, almost everywhere with respect to $\mu \Rightarrow \mu\{x\in Z:g(x)\leq\frac{1}{n}\}=0$, a contradiction to the relation $ (*) $ obtained earlier. 
\end{proof}
The following theorem determines the exact class of connected ideals in $ \mxa $ in the $ \mmui $-topology.
\begin{theorem}\label{t-2.8}
	The following statements are equivalent for an ideal $ J $ in the ring $ \mxa $.\begin{enumerate}
		\item $ J $ is a connected ideal in the $ \mmui $-topology.
		\item $ J \subseteq\liinf$.
		\item $ J\subseteq \imuxa $.
	\end{enumerate}
\end{theorem}
\begin{proof}
	$ (1)\Rightarrow (2) $. Let $ (1) $ be true. If possible let there exist $ f\in J $ such that $ f\notin\liinf $. Since $ 0\in J\cap \liinf $, it follows that $ J\cap \liinf $ and $ J\setminus \liinf $ are two non-empty disjoint open sets in the space $ J $. This follows from the clopenness of $ \liinf $ in the space $ \mxa $. This contradicts the connectedness of $ J $.\\$ (2)\Rightarrow (3) $. Let $ (2) $ be true. If possible let there exist $ f\in J $ such that $ f\notin \imuxa $. Then there exists $ g\in \mxa $ such that $ fg\notin \liinf $. But since $ J $ is an ideal in $ \mxa $, $ f\in J $ implies that $ fg\in J $ and hence by $ (2) $, $ fg\in \liinf $, a contradiction.\\
	$ (3)\Rightarrow(1) $. Let $ (3) $ be true. Then for each $ f\in J $, it follows from Theorem \ref{t-2.4} that $ \phi_f $ is a continuous map. Therefore we get $ J=\bigcup\limits_{f\in J}\phi_f(\R) $, a connected subset of $ \mxa $ in the $ \mmui $-topology.
\end{proof}
\section{$ \umui $-topology on $ \mxa $}

\begin{theorem}\label{t-3.1}
	If $ f\in\liinf $, then the map $ \phi_f:\R\rightarrow \mxa $ given by $ \phi_f(r)=rf $ is continuous if $ \mxa $ is equipped with the $ \umui $-topology.
\end{theorem}
\begin{proof}
	Consider the basic open set $ \umui(rf,\epsilon) $, where $ r\in \R $ and $ \epsilon>0 $. Since $ f\in\liinf $, there exists $ \delta>0 $, $ Z\in \zi, E\in \cal{A}, E\subseteq Z $ such that $ \mu(E)=0 $ and for each $x\in Z\setminus E$, $ |f(x)|\leq\delta $. Let $ \lambda=\frac{\epsilon}{2\delta} $. Then for each $ t\in(r-\lambda, r+\lambda) $, $ |\phi_f(t)-\phi_f(r)|=|t-r||f|\leq
	\lambda |f|\leq \lambda\delta $ on $ Z\setminus E $. This implies that $ \sup\limits_{x\in Z\setminus E}|\phi_f(t)(x)-\phi_f(r)(x)|\leq\frac{\epsilon}{2}<\epsilon $. This means that $ \phi_f(t)\in \umui(\phi_f(r),\epsilon)$   where $ |t-r|<\lambda $. Thus the function $ \phi_f $ becomes continuous at the point $ r\in \R $.
\end{proof}
Since the $ \umui $-topology on $ \mxa $ is weaker than the $ \mmui $-topology, the next result is an improvement of Theorem \ref{t-2.2}.
\begin{theorem}\label{t-3.2}
	The set $ \liinf $ is a clopen subset of $ \mxa $ in the $ \umui $-topology.
\end{theorem}
\begin{proof}
	Let $ f\in\liinf $. Then since $ \zi $ is closed under finite intersection and the union of two sets of $ \mu $-measure zero is again a set of $ \mu $-measure zero, it follows by straight forward arguments that $ \umui(f,\frac{1}{3})\subseteq\liinf $ from which we can say that $ f $ is an interior point of $ \liinf $ in the $ \umui $-topology. Thus $ \liinf $ is open in this topology. To show that $ \liinf $ is a closed subset of $ \mxa $ in the same topology, let $ f\in$ cl$_{\mxa}\liinf $. Then for each $ \epsilon>0 $ in $ \R $, $ \umui(f,\epsilon)\cap\liinf\neq \emptyset $. Select a function $ g\in\umui(f,\frac{1}{2})\cap \liinf $. Then there exists $ Z\in \zi, E\subseteq Z $ with $ \mu(E)=0 $ and $ \delta>0 $ such that for each $x\in Z\setminus E$, $ |g(x)|\leq \delta $ and $ |g(x)-f(x)|\leq\frac{1}{2}$. This implies that for each $ x\in Z\setminus E $, $ |f(x)|\leq\delta+\frac{1}{2} $, from which it follows that $ f\in\liinf $. Thus $ \liinf $ is closed in $ \mxa $ in the $ \umui $-topology.
\end{proof}
\begin{theorem}\label{t-3.3}
	$ \liinf $ is the component of $ 0 $ in $ \mxa $ in the $ \umui $-topology.
\end{theorem}
\begin{proof}
	It follows from Theorem \ref{t-3.1} that for each $ f\in\liinf $, $ \phi_f(\R) $ is a connected subset of $ \mxa $ in the $ \umui $-topology. As $ \liinf $ is a vector space over $ \R $, we can write $ \liinf = \displaystyle  \bigcup\limits_{f\in\liinf}\phi_f(\R)= $ the union of connected sets each of which contains the function $ 0 $. Thus $ \liinf $ is a connected subset of $ \mxa $ in the $\umui $-topology. It follws from Theorem \ref{t-3.2} therefore that $ \liinf $ is a maximal connected subset of $ \mxa $ containing $ 0 $ in the $ \umui $-topology. In other wards $ \liinf $ is the component of $ 0 $ in $ \mxa $ in this topology.
\end{proof}
It is a standard result in general topology that the component of a point in a topological space is contained in a quasicomponent of the same point. But we have the following result in our situation.
\begin{theorem}
	The component and quasicomponent of $ 0 $ in $ \mxa $ in the $ \umui $-topology are identical.
\end{theorem}
\begin{proof}
	It is sufficient to show in view of the preceeding observation that the quasicomponent of $ 0 $ in $ \mxa$ is contained in the component of $ 0 $. This is immediate because $ \liinf $ is a clopen subset of $ \mxa $ in the $ \umui $-topology containing $ 0 $ and the quasicomponent of $ 0 $ is the intersection of all clopen sets which contain $ 0 $.
\end{proof}
It is known that the component of $ 0 $ in a topological ring is an ideal in this ring (Theorem 2.6, \cite{ref6}). The following counterexample shows that in a ring with a topology  which is compatible only with the underlying additive group structure, the component of $ 0 $ may not be an ideal.
\begin{example}\label{e-3.4}
Let $ X=[0,1] $, $ \mathcal{A}= $ the $ \sigma $-algebra of all Lebesgue measurable subsets of $ X $, $ \mu $ is the Lebesgue measure. Thus $ \mxa $ is the ring of all real-valued Lebesgue measurable functions over $ [0,1] $. Suppose $ I=\{f\in \mxa: f[0,\frac{1}{3}]=\{0\}\} $. Then $ I $ is an ideal in $ \mxa $. We assert that $ \liinf\subsetneq \mxa $ and hence $ \liinf $ is not an ideal in $ \mxa $ because the function $ 1\in\liinf $.
\end{example}
\begin{proof}[Proof of the assertion:] Define the function $ f:[0,1]\rightarrow\R $ as follows: corresponding to a strictly increasing sequence $ t_1<t_2<\ldots<t_n<\ldots $ in the interval $ (0,\frac{1}{3}) $ with $\lim\limits_{n\rightarrow \infty}t_n=\frac{1}{3}  $, set
$ f([0,t_1))=\{1\} $, $ f([t_{n-1},t_n))=\{n\} $ for $ n\geq2, n\in\mathbb{N} $ and $ f=0 $ on $ [\frac{1}{3},1] $. We see that $ f\in\mxa $ (by the well-known pasting lemma for measurable functions) but $ f\notin\liinf $. Thus $  f\in\mxa\setminus \liinf $.
\end{proof}

We recall $ \umuip=\{f\in\mxa:$ there exists $Z\in\zi$ and $E\subseteq Z$ with $\mu(E)=0$ such that for each $x\in Z\setminus E,\ f(x)>0 \} $. We now set $ \umuipp=\{f\in\mxa:$ there exists $ Z\in\zi, E\subseteq Z$ with $\mu(E)=0$ and $\lambda>0$ in $ \R $ such that for each $ x\in Z\setminus E,\ f(x)\geq\lambda \} $. It is clear that $ \umuipp\subseteq\umuip $.
The following theorem decides when $ \umuip $ and $ \umuipp $ are equal. 
\begin{theorem}\label{t-3.5}
	For an ideal $ I $ in $ \mxa $, $ \umuip=\umuipp $ if and only if the $ \mmui $-topology on $ \mxa $ coincides with the $ \umui $-topology.
\end{theorem}
\begin{proof}
	First assume that $ \mmui $-topology = $ \umui $-topology. If possible let $ \umuipp \subsetneq\umuip$. So we can select a $ \omega $ from $ \umuip $ such that $ \omega\notin\umuipp $. Now for any $ \epsilon>0 $, there does not exist any $ Z\in\zi $ and an $ E\subseteq Z $ with $ \mu(E)=0 $ for which we can write $ |\omega(x)|\geq\frac{\epsilon}{3} $ for all $ x\in Z\setminus E $. This implies that for any $ f\in\mxa $, $ f+\frac{\epsilon}{3}\notin\mmui(f,\epsilon) $. On the other hand $ f+\frac{\epsilon}{3}\in \umui(f,\epsilon) $. Thus $ \umui(f,\epsilon)\nsubseteq \mmui(f,\epsilon) $. This shows that the set $ \mmui(f,\omega) $, which is open in the $ \mmui $-topology is not open in the $ \umui $-topology, a contradiction. Therefore $ \umuip=\umuipp $.
	
	To prove the converse, let $ \umuipp=\umuip $. Then for $ v\in \umuip $, there exists $ M>0,\ Z\in\zi $ and $ E\subseteq Z $ with $ \mu(E)=0 $ such that for each $ x\in Z\setminus E $, $ v(x)>M $. This implies that for any $ f\in\mxa $, $ f\in\umui(f,M)\subseteq \mmui(f,v) $. Thus each open set in the $ \mmui $-topology is open in the $ \umui $-topology and hence these two topologies on $ \mxa $ coincide.
\end{proof}
The follwing comprehensive theorem offers several conditions each necessary and sufficient for the coincidence of the $ \mmui $-topology and the $ \umui $-topology on $ \mxa $.
\begin{theorem}\label{t-3.6}
	The following statements are equivalent.
	\begin{enumerate}
		\item The $ \umui $-topology = the $ \mmui $-topology on $ \mxa $.
		\item $ \mxa $ with the $ \umui $-topology is a topological ring.
		\item $ \mxa $ with the $ \umui $-topology is a topological vector space.
		\item $ \liinf=\mxa $.
		\item $ \imuxa=\liinf $.
		\item $ \imuxa=\mxa $.
		\item $ \mxa $ with the $ \umui $-topology is connected.
		\item $ \zi $ is a $ \mu $-stable family.

	\end{enumerate}
\end{theorem}
\begin{proof}
$ (1)\Rightarrow (2) $ is immediate because $ \mxa $ with the $ \mmui $-topology is a topological ring. $ (2)\Rightarrow (3) $ is also immediate because each (real-valued) constant function on $ (X,\cal{A}) $ is measurable.\\
$ (2)\Rightarrow (4):$ Let $ (2)$ be true. Then since the component of $ 0 $ in a topological ring is an ideal in this ring (Theorem 2.6, \cite{ref6}), it follows from Theorem \ref{t-3.3}, that $ \liinf $ is an ideal in $ \mxa $. Since $ 1\in\liinf $, this implies that $ \liinf=\mxa $.\\
$ (4)\Rightarrow (5):$ Let $ (4) $ be true. Choose $ f\in\liinf=\mxa $. Then for each $ g\in \mxa $, $ g\in\liinf $ implies that $ fg\in\liinf $. Thus $ f\in\imuxa $. Hence $ \liinf \subseteq\imuxa\subseteq \liinf$. Therefore $ \liinf=\imuxa $.\\
$ (5)\Rightarrow (6) $ is immediate because $ \imuxa $ is an ideal in the ring $ \mxa $ and $ 1\in\liinf $.\\
$ (6)\Rightarrow (7) $ follows from Theorem \ref{t-2.5}, in conjunction with the fact that $ \mmui $-topology on $ \mxa $ is finer than the $ \umui $-topology.\\
$ (5)\Leftrightarrow (8) $ is the same as $ (2)\Leftrightarrow (6) $ in Theorem \ref{t-2.7}.\\
$ (7)\Rightarrow (4) $ follows directly from Theorem \ref{t-3.2} as the set $ \liinf $ is non-void.\\
$ (4)\Rightarrow (1) $. Let $ (4) $ be true. Then $ \mxa=\bigcup\limits_{\epsilon>0}\umui(0,\epsilon) $. Let $ u\in\umuip $. To prove the validity of $(1)$, it suffices to show in view of Theorem \ref{t-3.5} that $ u\in\umuipp $. Now $ u\in\umuip $ implies that $ u(x)>0 $ for each $ x\in Z(f)\setminus E $ for a suitable $ f\in I $ and an appropriate $ E\subseteq Z(f) $ with $ \mu(E)=0 $. Consequently then $ Z(u)\cap(Z(f)\setminus E)=\emptyset $. Define the function $ u^*:X\rightarrow \R $ as follows: 
 $$ u^*(x)=\begin{cases}
 	 u(x) & if\  x\in Z(u)\setminus E\\
 1 &  otherwise. 
 \end{cases}
$$
Then $ u^*\in\mxa $ and $ Z(u^*)\cap Z(f)=\emptyset $. This implies that $ Z(u^{*^2}+f^2)=\emptyset $. Thus $ \frac{1}{u^{*^2}+f^2} \in \mxa=\liinf$. Hence there exists $ \epsilon>0 $ such that $ \frac{1}{u^{*^2}+f^2}\in \umui(0,\epsilon) $. Therefore there exists $ Z\in\zi $ and $ E^*\subseteq Z $ with $ \mu(E^*)=0 $ such that $ \sup\limits_{x\in Z\setminus E^*} \frac{1}{(u^{*^2}+f^2)(x)}<\epsilon$. This implies that for each $ x\in Z\setminus E^* $, $ (u^{*^2}+f^2)(x)>\frac{1}{\epsilon} $ and hence for each $ x\in (Z\setminus E^*)\cap Z(f) $, $ u(x)>\frac{1}{\sqrt{\epsilon}} $. This implies that $ u\in\umuipp $. Thus $ \umuip=\umuipp $. Thus far we have established that the seven statements $ (1),(2), (4), (5), (6), (7), (8) $ are equivalent. Since $ (2)\Rightarrow (3) $ is already established, to complete this theorem it suffices to prove that $ (3)\Rightarrow (4) $. So assume that $ (4) $ is false. Then there exists an $ f\in \mxa $, $ f\geq 1 $ on $ X $ with the following property: for each $ Z\in\zi $ and each $ E\subseteq Z $ with $ \mu(E)=0 $ and for each $ n\in\mathbb{N},\ \mu\{x\in Z \setminus E: f(x)>n\}>0 $. We assert that the scalar multiplication: \begin{alignat*}{2}
	\R\times\mxa&\rightarrow \mxa\\
	(r,g)&\mapsto r.g
\end{alignat*}
is not continuous at the point $ (0,f) $ in the $ \umui$-topology and therefore $ (3) $ is false.\\

\end{proof}
\section{Question of completeness and separability of the pseudonormed linear space $ \liinf$}

At the very beginning of this section, we make two additional assumptions: \begin{enumerate}
	\item For the ideal $ I $ in $ \mxa $, the associated filter $ \zi=\{Z(f):f\in I\} $ of measurable sets is closed under countable intersection. We note that this condition is fulfilled if $ I $ is a real maximal ideal in $ \mxa $ and is trivially fulfilled with $ I=\{0\} $.
	\item For each $ E\in \zi, \mu(E)>0 $. Observe that this condition is also obeyed by $ I=\{0\} $ if $ \mu(X)>0 $.
\end{enumerate}
The following lemma will be helpful towards the completeness of $ \liinf $.
\begin{lemma}\label{l-4.1}
	For each $ f\in\liinf $, $ ||f||_I^{\infty} $ is itself an essential $ I $-upper bound of $ f $ over $ X $ $ ( $thus the set of all essential $ I $-upper bounds of $ f $ over $ X $ has a smallest member$ ) $.
\end{lemma}
\begin{proof}
	By a known property of infimum, there exists a decreasing sequence $ t_1>t_2>\ldots $ of positive real numbers with the following property: each $ t_n $ is an essential $ I $-upper bound of $ f $ and $\lim\limits_{n\rightarrow \infty}t_n=||f||_I^{\infty} $. Consequently for each $ n\in \mathbb{N} $, there exists $ Z_n\in\zi $ and a measurable subset $ E_n $ of $ Z_n $ such that $ \mu(E_n)=0 $ and for each $ x\in Z_n\setminus E_n $, $ |f(x)|\leq t_n $. Take $ Z=\displaystyle \bigcap_{n=1}^{\infty} Z_n$. Then by the additional hypothesis $ Z\in\zi $ and $ \{x\in Z:|f(x)|>||f||_I^{\infty}\} \subseteq \displaystyle \bigcup_{n=1}^{\infty}\{x\in Z_n:|f(x)|>t_n\}$. This implies that \[ \mu\{x\in Z:|f(x)|>||f||_I^{\infty}\}\leq \sum_{n=1}^{\infty} \mu\{x\in Z_n:|f(x)|>||f||_I^{\infty}\} \  \ldots (*)\] We recall that $ |f(x)|\leq t_n $ for each $ x\in Z_n\setminus E_n $ and $ \mu(E_n)=0 $. This implies that $ \mu\{x\in Z_n:|f(x)|>t_n\}=0 $ for each $ n\in \mathbb{N} $. It follows from the above relation $ (*) $ that $ \mu\{x\in Z:|f(x)|>||f||_I^{\infty}\}=0 $. Hence $ ||f||_I^{\infty} $ is an essential $ I $-upper bound of $ f $ over $ X $.
\end{proof}
\begin{theorem}\label{t-4.2}
	The pseudonormed linear space $ (\liinf, ||.||_I^{\infty}) $ is complete.
\end{theorem}
\begin{proof}
	Let $ \{f_k\} $ be a cauchy sequence in $ \liinf $. Since  by  Lemma \ref{l-4.1} for each pair of indices $ k,j\in \mathbb{N} $, $ ||f_k-f_j||_I^{\infty} $ is an essential $ I $-upper bound of the function $ f_k-f_j $, it follows that there exists a set $ Z_{k,j}\in \zi $ and a measurable set $ E_{k,j}\subseteq Z_{k,j} $ with $ \mu(E_{k,j})=0 $ for which we can write: for each $ x\in  Z_{k,j}\setminus E_{k,j}$, $ |f_k(x)-f_j(x)|\leq ||f_k-f_j||_I^{\infty} $. Let $ Z=\bigcap\limits_{k=1}^{\infty}\bigcap\limits_{j=1}^{\infty} Z_{k,j}$ and $ E=\bigcup\limits_{k=1}^{\infty}\bigcup\limits_{j=1}^{\infty}E_{k,j} $. Then by hypothesis $ Z\in \zi $ and $ \mu(E)=0 $. Thus we can write for each $ x\in Z\setminus E $, for each $ k,j\in\mathbb{N} $, \[\ \ \ \ \ \ \ \ \ \ \   |f_k(x)-f_j(x)|\leq ||f_k-f_j||_I^{\infty} \ \ \ \ \ \ \ \ \ \ \ \ \ \ \ldots (1) \] Since $ \{f_k\}$ is initially chosen to be a cauchy sequence in $ \liinf $, it follows from the relation $ (1) $ above that, for each $ x\in Z\setminus E $, $ \{f_k(x)\}_{k=1}^{\infty} $ is a cauchy sequence in $ \R $ and hence $\lim\limits_{k\rightarrow\infty}f_k(x)  $ exists in $ \R $. Define the function $ g:X\rightarrow \R $ as follows: $ g(x)= $ $\lim\limits_{k\rightarrow \infty}f_k(x)  $ if $ x\in Z\setminus E $ and $ g(x)=0 $ otherwise. Then $ g\in \mxa $. To complete this theorem, it remains to check that $ g\in\liinf $ and $\lim\limits_{k\rightarrow \infty}f_k=g $ in this space. Choose $ \epsilon>0 $. Then there exists $ k_0\in \mathbb{N} $ such that for all $j\geq k_0 $, for all $k\geq k_0 $, $ ||f_k-f_j||_I^{\infty}<\epsilon $. This inequality in conjunction with the inequality $ (1) $ yields:  for all $j\geq k_0 $, for all $k\geq k_0 $, for all $x\in Z\setminus E $, $ |f_k(x)-f_j(x)|<\epsilon $ on taking limit as $ k\rightarrow \infty $, this further yields that:
	 \begin{alignat*}{2}
	 	\forall\ x\in Z\setminus E,\ \forall\ j\geq & k_0,\ |\lim\limits_{k\rightarrow \infty}f_k(x)-f_j(x)|\leq\epsilon \\
	  &\text{i.e.,}\ |g(x)-f_j(x)|\leq\epsilon\ \ \ \ \  \ \ \ \ \ \ldots (2)
	 \end{alignat*}
	 It follows from $ (2) $ that $ g-f_{k_0} $ is essentially $ I $-bounded on $ X $, i.e., $ g-f_{k_0} \in\liinf$. Since $ f_{k_0}\in\liinf $, it follows therefore that $ g\in\liinf $. The relation $ (2) $ further implies that for all $ j\geq k_0 $, the positive number $\epsilon$ is an essential $ I $-upper bound of the function $ f_j-g\in\liinf $. It follows that $ ||f_j-g||_I^{\infty}\leq\epsilon$ for each $ j\geq k_0 $. Hence $\lim\limits_{j\rightarrow \infty}f_j=g $ in $ \liinf $. The proof is complete.
\end{proof}
Normally the space $ L^{\infty}(\mu)$ of all essentially bounded measurable functions over a measure space $ (X,\cal{A},\mu) $ is not separable (see the comments in page 580 in \cite{ref7}). It is well-known that the space of all essentially bounded Lebesgue measurable functions over the closed interval $ [0,1] $ is not separable \cite{ref9}. In our terminology this reads on choosing $ X=[0,1] $, $ \cal{A}= $ the $ \sigma $-algebra of all Lebesgue measurable subsets of $ [0,1] $, $ \mu= $ the Lebesgue measure and $ I=\{0\} $, that $ \liinf $ is not separable. It is also known that the space $ l^{\infty} $ of all bounded sequence of real numbers (or complex numbers) with familier supremum norm is non separable. We can realize $ l^{\infty} $ as the space $ \liinf $ on specialising $ X=\mathbb{N},\ \cal{A}=P(\mathbb{N}),\ \mu= $ the counting measure on $ \mathbb{N} $ and $ I=\{0\} $. We shall exhibit a non-trivial ideal $ I\neq\{0\} $ in the ring $ \mxa $ arising out in each of the above two examples to give two examples of spaces of the form $ \liinf $, none of which is separable.
\begin{example}\label{e-4.3}
	Let $ X=[0,1] $, $ \cal{A}= $ the $ \sigma $-algebra of all Lebesgue measurable subsets of $ [0,1] $, $ \mu= $ the Lebesgue measure and $$ I=\{f\in\mxa
	: f(x)=0\ \text{for each}\ x\in[0,\frac{1}{2}]\}.$$ Let $ t_1<t_2<\ldots $ be a strictly increasing sequence in the open interval $ (0,\frac{1}{2}) $ which converges to $ \frac{1}{2} $. Let $ A_n=(t_{n-1},t_n) $, $ n\geq 1 $, writing $ t_0=0 $. Suppose $ J=\{f\in\liinf: f(t_i)=0$ for all $i\in \mathbb{N}\cup\{0\},\ f=0\ \text{on}\ [\frac{1}{2},1] $ and for each $ n\in\mathbb{N} $, either $ f(A_n)=\{0\} $ or $ f(A_n)=\{1\} $\}. We assert that $ J $ is an uncountable set. To prove this assertion, if possible let $ J $ be a countable set: $ J=\{f_1,f_2,\ldots f_n,\ldots\} $. Define a function $ g:X\rightarrow \R $ as follows $ g(\{t_0,t_1,\ldots,t_n,\ldots\}\cup [\frac{1}{2},1])=0 $ and given $ n\in \mathbb{N} $, $ g(A_n)=1 $ if $ f_n(A_n)=\{0\} $ and $ g(A_n)=0 $ if $ f_n(A_n)=1 $. Then it is easy to check that $ g\in J $, but $ g $ is different from each $ f_n $ (on a set of positive Lebesgue measure), a contradiction. We further note that whenever $ f,g\in J $ and $ f\neq g $, then there exists $ n\in\mathbb{N} $ for which $ f(A_n)=\{0\} $ and $ g(A_n)=\{1\} $, say. This implies that $ g-f=1 $ on $ A_n $ and $ \mu(A_n)>0 $. It follows that no number $ <1 $ can be an essential $ I $-upper bound of $ g-f $ over $ [0,1]=X $. In other words $ ||f-g||_I^{\infty}\geq1 $. Indeed, it is easy to check that $ ||f-g||_I^{\infty}=1 $ if $ f,g\in J $ with $ f\neq g $. If for $ h\in\liinf $ and $ \epsilon>0 $ in $ \R $, $ B(h,\epsilon)=\{g\in\liinf:||g-h||_I^{\infty}\}= $ the open ball centered at `$ h $' with radius $\epsilon$ in the pseudonormed linear space $ \liinf $ then the last observation implies that $ \{B(f,\frac{1}{3}):f\in J\} $ is an uncountable family of pairwise disjoint open balls in this space. Thus $ \liinf $ turns out to be a non-separable pseudonormed linear space.
\end{example}
\begin{example}\label{e-4.4}
	Take $ X=\mathbb{N} $, $ \cal{A}= P(\mathbb{N}) =$ the power set of $ \mathbb{N} $, $ \mu= $ the counting measure on $ (X,\cal{A}) $ and $S= \{n_k\} $ be any subsequence of $ \mathbb{N} $, i.e., $ n_1<n_2<\ldots $ is an increasing sequence of natural numbers without any upper bound.
	
	Suppose $ I=\{f\in \mxa: f(n_k)=0$ for each $k\in \mathbb{N}\} $. Then $ I $ is an ideal in the ring $ \mxa $. In this setting $ \liinf $ consists precisely all those sequences $ \{x_n\} $ of real numbers for which the subsequences $ \{x_{n_k}\} $ are bounded.
	
	Let $J=\{f\in L_I^{\infty}(\mathbb{N},P(\mathbb{N}),\mu): f(n)=0$ if $n\in\mathbb{N}\setminus S$ and for each $n_k\in S,\ f(n_k)\in\{0,1$\}. Then $ J $ is an uncountable set and for any two distinct $ f,g\in J $, it is clear that there exists some $ n_k $ for which $ f(n_k)=1 $ and $ g(n_k)=0 $ (say). Then $ 1 $ is an essential $ I $-upper bound of $ f-g $ on $ \mathbb{N} $ because $ |f-g|\leq1 $ on the whole of $ S\in Z[I] $. On the other hand no number $ \lambda\ (0<\lambda<1) $ can be an essential $ I $-upper bound of $ g-f $ on $ \mathbb{N} $ as $ \{n\in \mathbb{N}:|(g-f)(x)|>\lambda\}=\{n\in\mathbb{N}:|g(x)-f(x)|=1\}\supseteq \{n_k\} $. This implies that $ \mu\{n\in \mathbb{N}:|(g-f)(x)|>\lambda\}\geq \mu\{n_k\}=1 $ because $ \mu $ is the counting measure on $ \mathbb{N} $. Thus $ ||g-f||_I^{\infty}=1 $. Now arguing as in the previous example, taking into consideration the uncountability of the set $ J $, we can conclude that $ L_I^{\infty}(\mathbb{N},P(\mathbb{N}),\mu) $ is a non-separable Banach space.
\end{example}
 We shall conclude this article after making a few remarks in connection with the following classification of ideals in the ring $ \mxa $, effected by the $ \umui $ and $ \mmui $-topologies, with two rather special choices of the measure $ \mu $, viz the counting measure and Dirac measure. With $ \mu= $ counting measure, let us call $ \umui $-topology to be simply $ U_I $-topology and the $ \mmui $-topology as $ m_I $-topology. Also for $ f\in\mxa $, $ u\in \umuip $ and $ \epsilon>0 $ in $ \R $, we prefer to write $ U_I(f,\epsilon) $ and $ m_I(f,u) $ instead of $ \umui(f,\epsilon) $ and $ \mmui(f,u) $ respectively. To address this problem and to make this paper self-contained, we write the following background material, which we reproduce from the article \cite{ref4}.
 
 Suppose $ \widehat{X} $ is an enlargement of the set $ X $ (in any convenient manner) with the purpose that it will serve as an index set for the aggregate of all maximal ideals in the ring $ \mxa $. For each measurable set $ A\in \cal{A} $, $ \bar{A}=\{p\in\widehat{X}: A\in \mathcal{U}^p\} $, here $ M^p $ is the maximal ideal in $ \mxa $ that corresponds to the point $ p\in \widehat{X} $ and $ \cal{U}=Z[M^p] $ is the associated ultrafilter of measurable sets on $ X $. If $ p\in X $, then $\cal{U}^p=\cal{U}_p=\{E\in \cal{A}:p\in E\}=$ the fixed ultrafilter of measurable sets on $ X $ corresponding to the point $ p $. Also $ \bar{A}\cap X=A $ for each $ A\in\cal{A} $. The family $ \{\bar{A}:A\in \cal{A}\} $ constitutes a base for the closed sets of some topology, familiarly known as the Stone-topology on $ \widehat{X} $. Using this topology on $ \widehat{X} $, a complete description of the maximal ideals in $ \mxa $ is provided by the following list $ \{M^p:p\in \widehat{X}\} $, where for $ p\in \widehat{X}$, $ M^p=\{f\in\mxa:p\in$ cl$_{\widehat{X}}Z(f)\}$. This may be called the measure theoretic analogue of Gelfand-Kolmogoroff theorem in the theory of rings of continuous functions. Since every ideal $ I $ in $ \mxa $ happens to be an intersection of maximal ideals \cite{ref4}, we can write $ I=\bigcap\limits_{p\in Y} M^p $ for some subset $ Y $ of $ \widehat{X} $.
 
 We now define two binary relations `$ \sim $' and `$ \approx $' on the set $ \cal{I} $ of all ideals in $ \mxa $ as follows: for $ I,J\in\cal{I} $, $ I\sim J $ if and only if $ m_I $-topology = $ m_J $-topology, and $ I\approx J $ if and only if $ U_I $-topology = $ U_J $-topology. Clearly `$ \sim $' and `$ \approx $' are equivalence relations on $ \cal{I} $.
 
 The following theorem essentially says that, these two classifications of ideals in $ \mxa $ via `$ \sim $' and `$ \approx $', albeit outwardly different are essentially the same. Indeed, each equivalence class in both the quotient sets $ \cal{I}/_{\sim} $ and $ \cal{I}/_{\approx} $ is a singleton.
\begin{theorem}\label{t-4.5}
	For two ideals $ I $ and $ J $ in $ \cal{I} $, $ I\sim J $ if and only if $ I\approx J $ if and only if $ I=J $.
\end{theorem}
\begin{proof}
	If $ I=J $, then it is trivial that $ I\sim J $ and $ I\approx J $. To prove the non-trivial part of this theorem assume that $ I\sim J $. we shall prove that $ J\subseteq I $, from symmetry it will follow that $ I\subseteq J $ and hence $ I=J $. If possible let $ J\nsubseteq I $. We can choose $ f\in J $ such that $f \notin I $. Now $ I=\bigcap\limits_{p\in Y}M^p $ for some subset $ Y $ of $ \widehat{X} $. Therefore $ f\notin M^q $ for some point $ q\in \widehat{X} $. It follows on using the above Gelfand-Kolmogoroff like theorem that $ q\notin$ cl$_{\widehat{X}}Z(f) $. But as $ \{\bar{A}:A\in\cal{A}\} $ constitutes a closed base for the Stone-topology on $ \widehat{X} $, this further implies that there exists $ A\in \cal{A} $ such that $ q\notin \bar{A} $ yet cl$_{\widehat{X}}Z(f)\subseteq \bar{A}  $. It follows that cl$_{\widehat{X}}Z(f)\cap X\subseteq \bar{A}\cap X$, this means that $ Z(f)\subseteq A $.
	
	Let $ h\in\mxa $ be the characteristic function of $ X\setminus A $, i.e., $ h=\chi_{X\setminus A} $. Now $ m_I(0,\frac{1}{2}) $ is open in $ \mxa $ in the $ m_I $-topology. To get the desired contradiction we shall check that $ m_I(0,\frac{1}{2}) $ is not open in the $ m_J $-topology and hence $ I $ is not related to $ J $ by the binary relation `$\sim $' on $ \cal{I} $. If possible let  $ m_I(0,\frac{1}{2}) $ be open in the $ m_J $-topology. Then there exists $ v\in U_{\mu_J}^+=U_J^+ $ such that $ m_J(0,v)\subseteq m_I(0,\frac{1}{2}) $. We note that $ v>0 $ on some $ Z_0\in Z[J] $. Now since $ f\in J $, if follows that $ Z(f)\cap Z_0\in Z[J] $. We recall that $ Z(f)\subseteq A $, from which it follows that $ h=0 $ on $ Z(f) $ and hence $ h=0 $ on $ Z(f)\cap Z_0\in Z[J] $. Since $ v $ is already `$ 0 $' on this last measurable set, this implies that $ h\in m_J(0,v)\subseteq m_I(0,\frac{1}{2}) $. Therefore there exists $ g\in I $ such that $ |h(x)|<\frac{1}{2} $ for each $ x\in Z(g) $ on the other hand $ I=\bigcap\limits_{p\in Y}M^p $ implies that $ g\in M^q $ and hence $ q\in  $ cl$_{\widehat{X}} Z(g)$. But as $ q\notin \bar{A} $, we get that $ \widehat{X}\setminus \bar{A} $ is an open neighbourhood of $ q $ in the space $ \widehat{X} $. This surely implies that $ (\widehat{X}\setminus \bar{A})\cap Z(g)\neq \emptyset $. Since $ Z(g)\subseteq X $, this further implies that $ (X\setminus A)\cap Z(g)\neq\emptyset $. Choose a point $ x\in (X\setminus A)\cap Z(g)$. Then $ x\in X\setminus A $ implies that $ h(x)=1 $. On the other hand $ x\in Z(g) $ implies that $ |h(x)|<\frac{1}{2} $, this is a contradiction. Thus far we have proved that $ I=J $ if and only if $ m_J $-topology = $ m_I $-topology on $ \mxa $.
	
	Now if we make a close scrutiny in the above arguments, then it is revealed that the assumption $ J\nsubseteq I $, made at the beginning of the proof of the present theorem, would have led to the conclusion that the open set $ U_I(0,\frac{1}{2}) $ in $ \mxa $ in the $ U_I $-topology (we write $ \umui(0,\frac{1}{2})=U_I(0,\frac{1}{2}) $) is not open in the $ m_I $-topology and consequently not open in $ U_I $-topology (because $ U_I $-topology $\subseteq$ $ m_I $-topology). Thus it is proved that $ I\approx J $ if and only if $ I=J $. The theorem is completely proved.
\end{proof}
Our next classification of ideals in $ \mxa $ is with respect to the choice $ \mu= $ an appropriately defined Dirac measure. So let us prepare our set up once again. Here $ X $ is an infinite set, $ \cal{A} =P(X)$, the family of all subsets of $ X $ and for any chosen $ x_0\in X $, the measure $ \delta_{x_0}:P(X)\rightarrow \{0,1\} $ is defined as follows: $  \delta_{x_0}(E)=\{1\} $ if $ x_0\in E $, $  \delta_{x_0}(E)=\{0\} $ if $ x_0\notin E $. $  \delta_{x_0} $ is called a Dirac measure on $ X $. As in the preceeding theorem $ \cal{I} $ stands for the set of all ideals in the measure space $ (X,P(X),  \delta_{x_0}) $. In view of the additional assumption (2) made at the beginning of this section, for any $ I\in \cal{I} $ and for each $ E\in\zi $, $  \delta_{x_0}(E)>0 $, this means that $ x_0\in E $. Thus $ x_0\in \bigcap Z[I] $. Now as in the preceeding theorem define the binary relation `$ \sim $' and `$ \approx $' on $ \cal{I} $ as follows: for $ I,J\in \cal{I} $, $ I\sim J $ if and only if $ m_{{ \delta_{x_0}}_I} $-topology = $ m_{{ \delta_{x_0}}_J} $-topology on $ \cal{M}(X,P(X),  \delta_{x_0}) $. Here we write $ m_{{ \delta_{x_0}}_I} $ in palce of $ \mmui $ on choosing $ \mu= \delta_{x_0} $. Also writing $ U_{{ \delta_{x_0}}_I}$-topology = $\umui $-topology with this choice of $ \mu $, we define our second binary relation `$ \approx $' on $ \cal{I} $ as follows: for $ I,J\in \cal{I} $, $ I\approx J $ if and only if $ U_{{ \delta_{x_0}}_I} $-topology = $ U_{{ \delta_{x_0}}_J} $-topology on $ \cal{M}(X,P(X),  \delta_{x_0}) $. It is easy to check that for any $ f\in \cal{M}(X,P(X)) $, $ u\in U_{{ \delta_{x_0}}_I}^+ $ and $ \epsilon>0 $ in $ \R $, $ m_{{ \delta_{x_0}}_I}(f,u)=m_{{ \delta_{x_0}}_J}(f,u) $ and $ U_{{ \delta_{x_0}}_I}(f,\epsilon) =U_{{ \delta_{x_0}}_J}(f,\epsilon)$ for any two ideals $ I $ and $ J $ in the family $ \cal{I} $. This leads to the following proposition which as something far opposite to what happens in Theorem \ref{t-4.5} that each of the quotient sets $ \cal{I}/_{\sim} $ and $ \cal{I}/_{\approx} $, degenerates into a one-membered set.
\begin{theorem}\label{t-4.6}
	For any two ideals $ I $ and $ J $ in $ \cal{I} $, $ I\sim J $ and also $ I\approx J $.
\end{theorem}
%The following problems are still open.
\begin{question}
The following problems are still open.
	\begin{enumerate}
		\item Suppose an ideal $ I $ in $ \mxa $ is such that the pseudonormed linear sapce $ \liinf $ is complete. Is the filter $ Z[I] $ closed under countable intersection?
		\item Does there exist a measure space $ (X,\cal{A},\mu) $ and a non-zero ideal $ I $ in $ \cal{M}(X,\cal{A},\mu) $ for which the pseudonormed linear space $ \liinf $ is separable?
	\end{enumerate}
\end{question}
	
\end{document}